\documentclass{amsart}
\usepackage{amssymb,amscd,amsthm}
\usepackage{latexsym}
\date{\today}

\usepackage[colorlinks=true]{hyperref}

\newcommand{\Z}{{\mathbb Z}}
\newcommand{\R}{{\mathbb R}}

\newcommand{\T}{{\mathbb T}}
\newcommand{\Q}{{\mathbb Q}}

\newtheorem{theorem}{Theorem}

\newtheorem{lemma}{Lemma}
\newtheorem{prop}{Proposition}
\newtheorem{coro}{Corollary}

\begin{document}
\title[Pinned Repetitions in Symbolic Flows]{Pinned Repetitions in Symbolic Flows:\\Preliminary Results}

\author[M.\ Boshernitzan]{Michael Boshernitzan}

\address{Department of Mathematics, Rice University, Houston, TX~77005, USA}

\email{michael@rice.edu}

\author[D.\ Damanik]{David Damanik}

\address{Department of Mathematics, Rice University, Houston, TX~77005, USA}

\email{damanik@rice.edu}

\thanks{D.\ D.\ was supported in part by NSF grant
DMS--0653720.}

\begin{abstract}
We consider symbolic flows over finite alphabets and study certain
kinds of repetitions in these sequences. Positive and negative
results for the existence of such repetitions are given for
codings of interval exchange transformations and codings of
quadratic polynomials.
\end{abstract}

\maketitle

\section{Introduction}

Fix some finite alphabet $\mathcal{A}$ and consider the compact
space $\mathcal{A}^\Z$ of two-sided sequences over this alphabet.
Here we endow $\mathcal{A}$ with the discrete topology and
$\mathcal{A}^\Z$ with the product topology.

We want to study repetitions in sequences $\omega \in
\mathcal{A}^\Z$ that either begin at the origin or are centered at
the origin and hence are ``pinned.'' Moreover, we also want to
study how many times arbitrarily long subwords are repeated in
this way. Thus, for $\omega \in \mathcal{A}^\Z$, we define
$$
R_n(\omega) = 1 + \frac{1}{n} \sup \{ m : \omega_k = \omega_{k+n}
\text{ for } 1 \le k \le m \},
$$
$$
T_n(\omega) = 1 + \frac{1}{n} \sup \{ m : \omega_k = \omega_{k+n}
\text{ and } \omega_{n+1-k} = \omega_{1-k} \text{ for } 1 \le k
\le m \},
$$
and
$$
R(\omega) = \limsup_{n \to \infty} R_n(\omega),
$$
$$
T(\omega) = \limsup_{n \to \infty} T_n(\omega).
$$
For definiteness, we will declare $\sup \emptyset = 0$. Notice
that $R(\omega), T(\omega)$ may be infinite and that we always
have $T(\omega) \le R(\omega)$.

Repetitions that begin at the origin have been studied, for
example, by Berth\'e et al.\ \cite{bhz}.  (In their terminology,
$R(\omega)=\mathrm{ice}(\omega)$,
the {\em initial critical exponent}\, of  $\omega$).

 Repetitions that are
centered at the origin are of interest in the study of
Schr\"odinger operators; compare the survey articles \cite{d1,d2}.
As explained there, if $T(\omega)$ is sufficiently large, one can
prove results about the continuity of spectral measures of an
associated Schr\"odinger operator using the Cayley-Hamilton
\mbox{theorem}.

Denote the shift transformation on $\mathcal{A}^\Z$ by $S$, that
is, $(S \omega)_k = \omega_{k+1}$. A subshift $\Omega$ is a
closed, $S$-invariant subset of $\mathcal{A}^\Z$. A subshift
$\Omega$ is called minimal if the $S$-orbit of every $\omega \in
\Omega$ is dense in $\Omega$. We denote by $\mathcal{W}(\omega)$
the set of all finite words over $\mathcal{A}$ that occur
somewhere in $\omega$. If $\Omega$ is minimal, then there is a set
$\mathcal{W}(\Omega)$ such that $\mathcal{W}(\omega) =
\mathcal{W}(\Omega)$ for every $\omega \in \Omega$. Finally, we
let $\mathcal{W}_n(\Omega) = \mathcal{W}(\Omega) \cap
\mathcal{A}^n$.

\begin{prop}
Suppose that $\Omega$ is a minimal subshift. Then,
$R_\mathrm{max}(\Omega) = \max_{\omega \in \Omega} R(\omega)$ and
$T_\mathrm{max}(\Omega) = \max_{\omega \in \Omega} T(\omega)$ both
exist, as elements of $[1,\infty]$. Moreover, the sets $\{ \omega
\in \Omega : R(\omega) = R_\mathrm{max}(\Omega) \}$ and $\{ \omega
\in \Omega : T(\omega) = T_\mathrm{max}(\Omega) \}$ are residual.
\end{prop}

\begin{proof}
We prove the two statements for two-sided pinned repetitions. The
one-sided case may be treated analogously.

Note that it suffices to show that, for any $\hat \omega \in
\Omega$, the set $M(\hat \omega) = \{ \omega \in \Omega :
T(\omega) \ge T(\hat \omega) \}$ is residual. Indeed, once this is
shown, one may choose a sequence in $\{ \hat \omega^{(k)} \}
\subset \Omega$ such that $T(\hat \omega^{(k)}) \to \sup_{\omega
\in \Omega} T (\omega)$ and then consider the residual set $M =
\bigcap_{k \ge 1} M(\hat \omega^{(k)})$. By construction,
$T(\omega) = \sup_{\omega \in \Omega} T (\omega)$ for every
$\omega \in M$, and hence the $\sup$ is a $\max$, and since the
set $\{ \omega \in \Omega : T(\omega) = T_\mathrm{max}(\Omega) \}$
contains $M$, it is residual.

So let $\hat \omega \in \Omega$ be given. Choose a finite or
countable strictly increasing sequence $t_1, t_2 , \ldots$ with
$T(\hat \omega) = \sup_m t_m$. Fix $m$. By minimality, for each
length $l$, there is $N_l$ such that every word in
$\mathcal{W}_{N_l}(\Omega)$ contains all words from
$\mathcal{W}_l(\Omega)$ as subwords. Thus, it is possible to find,
for each $w \in \mathcal{W}_{2n+1}(\Omega)$ a word $E(w,m) \in
\mathcal{W}(\Omega)$ of odd length that has $w$ as its central
subword of length $2n+1$ and obeys\footnote{While $T$ was defined
above only for two-sided infinite words, a completely analogous
definition can be given for finite words, where the central
position plays the role of the origin.}
\begin{equation}\label{rewmest}
T(E(w,m)) \ge t_m - \frac{1}{n}.
\end{equation}
We may simply shift $\hat \omega$ until the first occurrence of
$w$ is centered at the origin and then take a sufficiently long
finite piece that is centered at the origin as well. By the
consequence of minimality mentioned above, it is clearly possible
to ensure an estimate of the form \eqref{rewmest} that is uniform
in $w \in \mathcal{W}_{2n+1}(\Omega)$.

Consider the open set
$$
\bigcup_{w \in \mathcal{W}_{2n+1}(\Omega)} [E(w,m)],
$$
where, for a word $x$ of odd length, $[x] = \{ \omega \in \Omega :
\omega_{-\frac{|x| - 1}{2}} \ldots \omega_{\frac{|x| - 1}{2}} = x
\}$. Notice that
$$
\mathcal{T}_m = \bigcup_{n \ge m} \; \bigcup_{w \in
\mathcal{W}_{2n+1}(\Omega)} [E(w,m)],
$$
is dense in $\Omega$. Consequently,
$$
\mathcal{T} = \bigcap_{m \ge 1} \mathcal{T}_m
$$
is a dense $G_\delta$ subset of $\Omega$ and for every $\omega \in
\mathcal{T}$, we have by construction $T(\omega) \ge T(\hat
\omega)$.
\end{proof}

\begin{prop}\label{erg}
Suppose that $\mu$ is an $S$-ergodic probability measure on
$\Omega$. Then there exist $R_\mu$ and $T_\mu$ such that
$R(\omega) = R_\mu$ and $T(\omega) = T_\mu$ for $\mu$-almost every
$\omega \in \Omega$.
\end{prop}

\begin{proof}
Clearly, $T(\cdot)$ is invariant and hence it is $\mu$-almost
surely constant. While $R(\cdot)$ may not be globally invariant,
we always have the inequality $R(S \omega) \ge R(\omega)$ for
every $\omega \in \Omega$. This implies that $R(S \omega) =
R(\omega)$ for $\mu$-almost every $\omega$ and hence $R(\omega)$
is $\mu$-almost surely constant.
\end{proof}

\noindent\textit{Remark.} There are some related results in
\cite{bhz}. They prove that, for minimal subshifts, $R(\cdot)$
attains its maximum (without showing that it does so on a residual
set). They also establish the almost sure constancy of $R(\cdot)$
with respect to any ergodic measure under the assumption of
minimality and sublinear block complexity  (\cite{bhz}, Proposition 2.1).
Our result (Proposition \ref{erg}) holds in complete generality,
and its proof is very short.

\medskip

In the remainder of the paper, we study pinned repetitions in
symbolic flows that are generated by coding certain specific
transformations of finite-dimensional tori. In fact, given the
space allotment, we will focus on two such classes -- interval
exchange transformations and quadratic polynomials arising in the
study of skew-shifts. We intend to continue our study of pinned
repetitions in symbolic flows in a future work. We would also like
to point out that the present study is related to our recent
papers \cite{bd1,bd2} on the repetition property for dynamical
systems on general compact metric spaces
(which are not necessarily totally disconnected).

\section{Preliminaries}

In this section, we present several results that will be useful
later in our study of pinned repetitions in certain specific
models.

\subsection{Irrational Rotations of the Circle}

For $x \in \R$, we write $\langle x \rangle = \mathrm{dist} \, (x
, \Z)$. Notice that $d(x,y) = \langle x - y \rangle$ is a metric
on $\T = \R / \Z$.

Recall that the Farey sequence of order $n$ is the sequence of
reduced fractions between $0$ and $1$ which have denominators less
than or equal to $n$, arranged in order of increasing size. Thus,
for example, $F_1 = \left\{ \frac01, \frac11 \right\}$, $F_2 =
\left\{ \frac01, \frac12, \frac11 \right\}$, $F_3 = \left\{
\frac01, \frac13, \frac12, \frac23, \frac11 \right\}$, $F_4 =
\left\{ \frac01, \frac14, \frac13, \frac12, \frac23, \frac34,
\frac11 \right\}$. If $\frac{a}{b}$ and $\frac{c}{d}$ are
neighbors in a Farey sequence, then $b+d > n$ and $|\frac{a}{b} -
\frac{c}{d}| = \frac{1}{bd}$; see, for example, \cite{hw}.

\begin{lemma}\label{fareylemma}
Let $x,y \in \T$ and $u > 0$ be such that for some $n \ge 2$, the
set $\{ x+y, 2x+y, 3x+y, \ldots, nx + y \}$ does not intersect an
arc $J \subseteq \T$ of length $|J| = u$. Then for some positive
integer $q < \min \left\{ n , \frac{2}{u} \right\}$, we have
$\langle q x \rangle < \frac{1}{n}$.
\end{lemma}

\begin{proof}
Of course, we can assume without loss of generality that $y = 0$.
Denote $S = \left\{ kx : 1 \le k \le \left\lfloor \frac{2}{u}
\right\rfloor \right\}$ and consider the position of $x$ relative
to the Farey sequence of order $n-1$: $\frac{r_1}{s_1} \le x \le
\frac{r_2}{s_2}$. By the well-known properties of the Farey
sequence of order $n-1$ recalled above, it follows that $s_j < n$
for $j = 1,2$ and $s_1 + s_2 \ge n$.

The point $\frac{r_1 + r_2}{s_1 + s_2}$ subdivides the interval
$[\frac{r_1}{s_1} , \frac{r_2}{s_2}]$. Suppose $x \in
[\frac{r_1}{s_1} , \frac{r_1 + r_2}{s_1 + s_2}]$. Then,
$$
\left| x - \frac{r_1}{s_1} \right| \le \frac{r_1 + r_2}{s_1 + s_2}
- \frac{r_1}{s_1} = \frac{1}{s_1(s_1 + s_2)} < \frac{1}{s_1 n }.
$$
The case $x \in [\frac{r_1 + r_2}{s_1 + s_2} , \frac{r_2}{s_2}]$
is analogous. Consequently, for $\frac{p}{q}$ equal to either
$\frac{r_1}{s_1}$ or $\frac{r_2}{s_2}$, we have that
\begin{equation}\label{xrsest}
\left| x - \frac{p}{q} \right| < \frac{1}{q n}.
\end{equation}

In particular, we must have $q < n$ and $\langle q x \rangle <
\frac{1}{n}$. Moreover, since $\frac{p}{q}$ is a reduced fraction,
the estimate \eqref{xrsest} implies that the maximal gap (in $\T$)
of $\{ kx : 1 \le k \le q \}$ is bounded from above by
$\frac{2}{q}$. By assumption, we therefore must have $u <
\frac{2}{q}$, that is, $q < \frac{2}{u}$.
\end{proof}

\subsection{Continued Fraction Expansion}

Let us recall some basic results from the theory of continued
fractions; compare \cite{hw,Lang}. Given an irrational number
$\alpha \in \T$, there are uniquely determined $a_n \in \Z_+$, $n
\ge 1$ such that
$$
\alpha = \cfrac{1}{a_1+ \cfrac{1}{a_2 + \cfrac{1}{a_3 + \cdots}}}
\, .
$$
Truncation of this infinite continued fraction expansion after $k$
steps yields the $k$-th \textit{convergent} $\frac{p_k}{q_k}$. We
have the following two-sided estimate for the quality of
approximation of $\alpha$ by the $k$-th convergent:
$$
\frac{1}{q_k ( q_k + q_{k+1} )} < \left| \alpha - \frac{p_k}{q_k}
\right| < \frac{1}{q_k q_{k+1}}.
$$
The numerators and denominators of the convergents obey the
following recursive relations:
\begin{eqnarray*}
p_0 = 0, & p_1 = 1, & p_k = a_k p_{k-1} + p_{k-2} \text{ for } k
\ge 2, \\
q_0 = 1, & q_1 = a_1, & q_k = a_k q_{k-1} + q_{k-2} \text{ for } k
\ge 2.
\end{eqnarray*}

\subsection{Discrepancy Estimates for Quadratic Polynomials}

In this subsection we discuss uniform distribution properties of
quadratic polynomials. More precisely, given $\alpha,\beta,\gamma
\in \T$, we consider the points
\begin{equation}\label{qpxn}
x_n = \alpha n^2 + \beta n + \gamma \in \T.
\end{equation}
The numbers
$$
D_N = D_N(\alpha,\beta,\gamma) = \sup_{\mathrm{intervals} \, I
\subseteq \T} \; \left| \frac{1}{N} \# \{ n : 1 \le n \le N , \;
x_n \in I \} - \mathrm{Leb}(I) \right|
$$
measure the quality of uniform distribution of the given sequence
and are called its \textit{discrepancy}.

\begin{theorem}\label{qpdiscest}
Suppose $\alpha \in \T$ is irrational and
$$
\alpha = \frac{p}{q} + \frac{\theta}{q^2},
$$
where $p$ and $q$ are relatively prime and $|\theta| \le 1$. Then,
for $\beta, \gamma \in \T$ and $\varepsilon > 0$ arbitrary, we
have
$$
D_q(\alpha,\beta,\gamma) < C_\varepsilon q^{- \frac{1}{3} -
\varepsilon}
$$
with some constant $C_\varepsilon$ that only depends on
$\varepsilon$. In particular, the set $\{x_1 , \ldots , x_q\}$
intersects every interval $I \subseteq \T$ of length at least
$C_\varepsilon q^{- \frac{1}{3} - \varepsilon}$.
\end{theorem}

We will use the following version of the Erd\"os-Tur\'an Theorem,
which holds in fact for arbitrary real numbers $x_1 , \ldots ,
x_N$; compare \cite[pp.~112--114]{KN}.

\begin{theorem}[Erd\"os-Tur\'an 1948]
There is a universal constant $\tilde C$ such that for every $m
\in \Z_+$,
$$
D_N \le \tilde C \left( \frac{1}{m} + \sum_{h = 1}^m \frac{1}{h}
\left| \frac{1}{N} \sum_{n = 1}^N e^{2 \pi i h x_n} \right|
\right).
$$
\end{theorem}

This estimate relates discrepancy bounds to bounds for exponential
sums. Thus the following lemma, which is closely related to a
lemma given on p.~43 of \cite{Lang}, is of interest.

\begin{lemma}\label{langlemma}
Suppose $\alpha \in \T$ is irrational, $\beta,\gamma \in \T$ are
arbitrary, and $x_n$ is given by \eqref{qpxn}. Then, we have for
$N \in \Z_+$,
$$
\left| \sum_{n = 1}^N e^{2 \pi i x_n} \right|^2 \le N +
\sum_{n=1}^N \min \left( 2N , \frac{1}{2 \langle n \alpha \rangle}
\right).
$$
\end{lemma}

\begin{proof}
This can be proved by a slight variation of the argument given on
pp.~43--44 of \cite{Lang}.
\end{proof}

\begin{proof}[Proof of Theorem \ref{qpdiscest}.]
By the Erd\"os-Tur\'an Theorem, we have
$$
D_q \lesssim \frac{1}{m} + \sum_{h = 1}^m \frac{1}{h} \left|
\frac{1}{q} \sum_{n = 1}^q e^{2 \pi i h x_n} \right|
$$
for every $m \in \Z_+$.\footnote{We write $a \lesssim b$ for
positive $a,b$ if there is a universal constant $C$ such that $a
\le C b$.} Take $m = \lfloor q^\delta \rfloor$ for some $\delta
\in (0,1)$. Then, together with Lemma~\ref{langlemma}, we find for
any $\varepsilon
0$,
\begin{align*}
D_q & \lesssim q^{-\delta} + \sum_{h = 1}^{\lfloor q^\delta
\rfloor} \frac{1}{h} \left| \frac{1}{q} \sum_{n = 1}^q e^{2 \pi i
h x_n} \right| \\
& \le q^{-\delta} + \frac{1}{q} \sum_{h = 1}^{\lfloor q^\delta
\rfloor} \frac{1}{h} \left( q + \sum_{n=1}^q \min \left( 2q ,
\frac{1}{2 \langle h n \alpha \rangle} \right) \right)^{1/2} \\
& \lesssim q^{-\delta} + \frac{1}{q} \sum_{h = 1}^{\lfloor
q^\delta \rfloor} \frac{1}{h} \left( q + \frac{h}{\varepsilon}
\left( 1 + q^{\delta - 1} \right) q^{1+\varepsilon} \right)^{1/2}
\\
& \lesssim q^{-\delta} + q^{\frac{1}{2} \delta - \frac{1}{2} +
\frac{\varepsilon}{2}}.
\end{align*}
Here, we applied \cite[Eqn.~(153) on p.~75]{Korobov} in the third
step. In the last line, the exponents coincide when $\delta =
\frac{1 - \varepsilon}{3}$.
\end{proof}

Given an irrational $\alpha$, Theorem~\ref{qpdiscest} gives a
discrepancy estimate for those values of $N$ that appear as
denominators in the convergents associated with $\alpha$. This
will be sufficient for our purpose. In some cases (e.g.,
$\alpha$'s of Roth type), it is possible to use \cite[Theorem~6 on
p.~45]{Lang} to modify the proof so as to cover all values of $N$.
Moreover, if one is only interested in a metric result, the
following improved discrepancy estimate is of relevance.

\begin{theorem}
Suppose $\alpha \in \T$ is irrational. Then, for Lebesgue almost
every $\beta \in \T$ and every $\gamma \in \T$, we have
$$
D_N(\alpha,\beta,\gamma) < C_\varepsilon N^{-\frac12} (\log
N)^{\frac{5}{2} + \varepsilon}.
$$
\end{theorem}

Since we will not use this result, we do not prove it. It may be
derived from \cite[Theorem~1.158 on p.~157]{DT}.

\section{Pinned Repetitions in Codings of Interval Exchange Transformations}

In this section we study pinned repetitions occurring in codings
of interval exchange transformations. Interval exchange
transformations are maps from an interval to itself that are
obtained by partitioning the interval and then permuting the
subintervals.

More explicitly, let $m > 1$ be a fixed integer and denote
$$
\Lambda_m = \{ \lambda \in \R^m : \lambda_j > 0 , \; 1 \le j \le m
\}
$$
and, for $\lambda \in \Lambda_m$,
$$
\beta_j(\lambda) = \begin{cases} 0 & j = 0 \\ \sum_{i=1}^j
\lambda_i & 1 \le j \le m \end{cases} , \; I^\lambda_j =
[\beta_{j-1}(\lambda),\beta_j(\lambda)), \; |\lambda| = \sum_{i =
1}^m \lambda_i, \; I^\lambda = [0,|\lambda|).
$$
Denote by $\mathcal{S}_m$ the group of permutations on
$\{1,\ldots,m\}$, and set $\lambda_j^\pi = \lambda_{\pi^{-1}(j)}$
for $\lambda \in \Lambda_m$ and $\pi \in \mathcal{S}_m$. With
these definitions, the $(\lambda,\pi)$-interval exchange map
$T_{\lambda,\pi}$ is given by
$$
T_{\lambda,\pi} : I^\lambda \to I^\lambda , \; x \mapsto x -
\beta_{j-1}(\lambda) + \beta_{\pi(j) -1}(\lambda^\pi) \text{ for }
x \in I_j^\lambda , \; 1 \le j \le m.
$$
A permutation $\pi \in \mathcal{S}_m$ is called irreducible if
$\pi(\{ 1,\ldots ,k\}) = \{ 1,\ldots ,k\}$ implies $k = m$. We
denote the set of irreducible permutations by $\mathcal{S}_m^0$.

Veech proved the following theorem in \cite{v}.

\begin{theorem}[Veech 1984]\label{veechthm}
Let $\pi \in \mathcal{S}_m^0$. For Lebesgue almost every $\lambda
\in \Lambda_m$ and every $\varepsilon > 0$, there are $N \ge 1$
and an interval $J \subseteq I^\lambda$ such that
\begin{itemize}

\item[(i)] $J \cap T^l J = \emptyset$, $1 \le l < N$,

\item[(ii)] $T$ is linear on $T^l J$, $0 \le l < N$,

\item[(iii)] $| \bigcup_{l = 0}^{N-1} T^l J | > (1 - \varepsilon)
|\lambda|$,

\item[(iv)] $| J \cap T^N J | > (1 - \varepsilon) |J|$.

\end{itemize}
\end{theorem}

An interval exchange transformation   $T=T_{\lambda,\pi}$
is said to satisfy property  V   if  for every $\varepsilon > 0$,
there are $N \ge 1$  and an interval $J \subseteq I^\lambda$ such that
the  four conditions of Theorem \ref{veechthm}  are satisfied.
This convention is motivated by the following result.

\begin{theorem}\label{v2thm}
Let $T=T_{\lambda,\pi}$  satisfy property V.  Then the coding
 $s$ of the $T$-orbit of
Lebesgue almost every $x \in I^\lambda$ with respect to any finite
partition of $I^\lambda$ obeys $R(s) = T(s) = \infty$.
\end{theorem}

\begin{proof}
Fix any finite partition $I^\lambda = J_1\sqcup \cdots \sqcup J_N$.

If $\limsup_{\varepsilon \to 0+} N(\varepsilon) < \infty$, it is
not hard to see that $T_{\lambda,\pi}$ is a ``rational rotation''
and the assertion of the theorem holds trivially.

Consider the other case and let $\varepsilon_n \to 0$ be such that
$N(\varepsilon_n) \to \infty$. By passing to a suitable
subsequence $\{\varepsilon_{n_m}\}$, we can ensure that the
Lebesgue measure of those points $x \in I^\lambda$ that do not
have $m$ repetitions in both directions is bounded by $2^{-m}$.
Thus, by Borel-Cantelli, almost every point has unbounded
repetitions in both directions, which implies the assertion.
\end{proof}

\begin{coro}
Let $\pi \in \mathcal{S}_m^0$. For Lebesgue almost every $\lambda
\in \Lambda_m$, the coding $s$ of the $T_{\lambda,\pi}$-orbit of
Lebesgue almost every $x \in I^\lambda$ with respect to any finite
partition of $I^\lambda$ obeys $R(s) = T(s) = \infty$.
\end{coro}

\section{Pinned Repetitions in Codings of Quadratic Polynomials}

In this section we study sequences $s$ of the following type.
Suppose $\{ J_l : 1 \le l \le N \}$ is a partition of $\T$ into
finitely many intervals, $(\alpha,\beta,\gamma) \in \T^3$, and $s$
denotes the coding of $\alpha n^2 + \beta n + \gamma$ with respect
to the partition $J$. For example, assigning distinct numbers
$\lambda_l$ to the partition intervals, we may write
$$
s_n = \sum_{l = 1}^N \lambda_l \chi_{J_l}(\alpha n^2 + \beta n +
\gamma).
$$
Sometimes we make the dependence of $s$ on the parameters explicit
and write $s(\alpha,\beta,\gamma)$ or $s(\alpha,\beta,\gamma,J)$.
As explained in \cite{f}, there is close connection between
codings of quadratic polynomials and codings of orbits of the
skew-shift on $\T^2$.

We are interested in identifying the numbers $R(s)$ and $T(s)$ for
such sequences $s$. We present a number of results regarding this
problem. Roughly speaking, these numbers may take on the extreme
values $1$ and $\infty$ and they grow with the quality with which
$\alpha$ can be approximated by rational numbers.

\subsection{Absence of Repetitions} We first consider the case
where $\alpha$ is not well approximated by rational numbers and
show that there indeed are no repetitions in the sense that $R(s)
= T(s) = 1$. To make the argument more transparent, we begin by
considering $\alpha$'s with bounded partial quotients. This means
that the coefficients $\{a_n\}$ in the continued fraction
expansion are bounded; equivalently, $\inf_{q \in \Z_+} q \langle
q \alpha \rangle > 0$. The set of such $\alpha$'s has zero
Lebesgue measure.

\begin{theorem}\label{propbadlyapprox}
Suppose that $\alpha \in \T$ has bounded partial quotients and
each partition interval $J_l$ has length strictly less than $1/2$.
Then, $R (s(\alpha,\beta,\gamma,J)) = T(s(\alpha,\beta,\gamma,J))
= 1$ for every $(\beta , \gamma) \in \T^2$.
\end{theorem}

\begin{proof}
Assume that $R(s) > 1 + \nu$ for some $\nu \in (0,1)$. Then, $s$
has infinitely many $(1 + \nu)$-repetitions starting at the
origin. Let $n$ be the length of such a prefix of $s |_{\Z_+}$
that is $(1 + \nu)$-repeated.

For $1 \le k \le \nu n$, write
\begin{align*}
y_k & = \alpha k^2 + \beta k + \gamma \\
z_k & = \alpha (n+k)^2 + \beta (n+k) + \gamma \\
d_k & = z_k - y_k = \alpha n^2 + \beta n + 2 \alpha n k.
\end{align*}
Choose  $0 < l < 1/2$ such that each interval of the partition
under consideration has length less than $l$. Then, for $1 \le k
\le \nu n$, $y_k$ and $z_k$ must fall in the same interval of the
partition. In particular, $\langle d_k \rangle$ is bounded above
by $l$ for each such $k$. Consequently, $d_k$ avoids an arc $J
\subseteq \T$ of length $u = 1 - 2l > 0$.

Applying Lemma~\ref{fareylemma} with $x = 2 \alpha n$ and $y =
\alpha n^2 + \beta n$, we find that there is a positive integer $q
< \frac{2}{u}$ such that $\langle q x \rangle < \frac{1}{n}$. On
the other hand, there is $c = c(u,\alpha) > 0$ such that $ \langle
q x \rangle = \langle 2 \alpha q n \rangle > \frac{c}{n}$ since
$\alpha$ has bounded partial quotients. Combining the two
estimates, we have that
\begin{equation}\label{sxuplo}
\frac{c}{n} < \langle q x \rangle < \frac{1}{n}.
\end{equation}

Now assume in addition that $n > \frac{2}{\nu u}$ and let $m =
\lfloor \frac{\nu n}{q} \rfloor$, so that $1 \le \frac{\nu n}{q}
\le m \le \nu n$. Consider the points $\{ d_{qk} : 1 \le k \le m
\}$. It follows from \eqref{sxuplo} that the diameter of this set
belongs to the interval $(\frac{c \nu}{q}, \frac{\nu}{q})$.

For $n$ sufficiently large, we obtain a contradiction because, as
we saw above, the points $y_{sk}$ are well distributed on $\T$ and
will come close to the partition points. Addition of the
difference $d_{qk}$ for $k$ with $1 \le k \le m$ sufficiently
large will then go ``across'' such a partition point, which
contradicts the assumption the $y_{qk}$ and $z_{qk}$ belong to the
same interval of the partition. It follows that $R(s) = 1$, which
also yields $T(s) = 1$.
\end{proof}

As pointed out above, this result covers only a set of $\alpha$'s
that has zero Lebesgue measure. Let us extend it to a larger set.
Recall that $\alpha \in \T$ is a Roth number if for every
$\varepsilon > 0$, there is a constant $c(\varepsilon)$ such that
$
\langle q\alpha \rangle > \frac{c(\varepsilon)}{q^{1 +
\varepsilon}},
$
for every $q \in \Z_+$.

 If we replace the qualitative
well-distribution property with the quantitative discrepancy bound
established above, virtually the same argument proves the
following result, which covers a set of $\alpha$'s that has full
Lebesgue measure.

\begin{theorem}\label{proprothnumber}
Suppose that $\alpha \in \T$ is a Roth number and each partition
interval $J_l$ has length strictly less than $1/2$. Then, $R
(s(\alpha,\beta,\gamma,J)) = T(s(\alpha,\beta,\gamma,J)) = 1$ for
every $(\beta , \gamma) \in \T^2$.
\end{theorem}

\begin{proof}
The only change that needs to be made to the argument above is the
following. Assuming $\alpha$ to be Roth, instead of
\eqref{sxuplo}, we can prove
\begin{equation}\label{sxuplo2}
\frac{c_\varepsilon}{n^{1+\varepsilon}} < \langle q x \rangle <
\frac{1}{n}
\end{equation}
for any $\varepsilon > 0$. Applying Theorem~\ref{qpdiscest}, we
see that for every $\tilde \varepsilon > 0$, the points $x_1 , x_2
, \ldots , x_n$ are $C_{\tilde \varepsilon} q_k^{- \frac{1}{3} -
\tilde \varepsilon}\!$-dense in $\T$, where $k$ is chosen such
that $q_k \le n < q_{k+1}$. Since the Roth condition also implies
that the $q_k$'s associated with $\alpha$ obey $q_{k+1} \lesssim
q_k^{1+\varepsilon}$, it follows that $x_1 , x_2 , \ldots , x_n$
is $C_{\tilde \varepsilon} n^{- (1+\varepsilon)(\frac{1}{3} +
\tilde \varepsilon)}$-dense in $\T$.

Now we can conclude the proof as before by considering the points
$\{ d_{qk} : 1 \le k \le m \}$. The addition of one of them to the
corresponding $y_{qk}$ will take the point across a partition
point by the estimates just obtained.
\end{proof}

\subsection{Infinite Repetitions} Let us now turn to the other
extreme and start off by studying the case of rational $\alpha$.
We will see that in this case, there are infinite repetitions for
almost every pair $(\beta,\gamma)$. The next step will then be to
identify situations in which we have infinite repetitions for
irrational $\alpha$'s that are well approximated by rational
numbers in a suitable sense.

Denote  by  $\T_w$  the subset of irrational numbers  in  $\T$
with unbounded partial quotients.  It is well known that  $\T_w$
is a set of full Lebesgue measure in  $\T$.

First we address the case  $\alpha=0$.

\begin{lemma}\label{lem0}
Suppose that  $\beta\in\T_w$.  Then,  for Lebesgue almost every
$\gamma\in \T$,
$$R(s(0,\beta,\gamma,J)) = T(s(0,\beta,\gamma,J)) = \infty.$$
\end{lemma}
\begin{proof}
In this case  $s=s(0,\beta,\gamma,J)$  is a coding of $\beta
n+\gamma$ with respect to the fixed partition  $J$.  The claim of
the lemma follows from the fact that the 2-interval exchange
transformation $T=T_{\lambda,\pi}$  with  $\lambda=(1-\beta,
\beta),\ \pi=(21)$ is Veech if  $\beta\in\T_w$ (see
Theorems~\ref{veechthm} and \ref{v2thm} above).

For an alternative, more direct argument, consult \cite{dp}.
\end{proof}

\begin{theorem}\label{ratreplem}
Suppose that $\alpha \in \T$ is rational.  Then, for every  $\beta\in\T_w$,
\begin{equation}\label{eq:rat}
R(s(\alpha,\beta,\gamma,J)) = T(s(\alpha,\beta,\gamma,J)) = \infty
\end{equation}
for Lebesgue almost every $\gamma \in \T$. In particular,
\eqref{eq:rat} holds for Lebesgue almost every
$(\beta,\gamma)\in\T^2$.
\end{theorem}

\begin{proof}
Write $\alpha = \frac{p}{q}$. Notice that $\alpha n^2$ is
$q$-periodic. This serves as a motivation to begin with a study of
repetitions along arithmetic progressions of step-length $q$. Put
differently, we regard $\beta (qn)$ as $(\beta q) n$ and then add
the constant $\alpha (qn)^2 + \gamma$.

For every  $\beta \in \T_w$, $\beta q\in T_w$ as well.  By
Lemma \ref{lem0},  there exists a sequence $Q_k \to \infty$ and
a set $\tilde G_\beta \subset \T$ of
full measure such that for $\tilde \gamma \in \tilde G_\beta$, the
coding of $(\beta q)n + \tilde \gamma$ with respect to the given
partition of $\T$ has $k$ repetitions of length $Q_k$ to both
sides, for every $k \ge 1$.

To piece together these repetitions along arithmetic progressions
of step-length $q$, define
$
\mathcal{G}'_\beta = \bigcap_{k = 1}^\infty \bigcap_{n =
1}^{Q_k} \{ \gamma \in \T : \alpha (qn)^2 + \beta (qn) + \gamma
\in \tilde G_\beta \}.
$

As a countable intersection of sets of full measure,
$\mathcal{G}'_\beta$ has full measure. We find that for
$\alpha = \frac{p}{q}$ rational, $\beta\in\T_w$  and
$\gamma \in\mathcal{G}'_\beta$  (and hence for almost every
$(\beta,\gamma) \in \T^2$), $R(s) = T(s) = \infty$.
\end{proof}

\begin{theorem}\label{propwellapprox}
There is a dense $G_\delta$ set $\mathcal{R} \subset \T$ such that
for $\alpha \in \mathcal{R}$, we have $R
(s(\alpha,\beta,\gamma,J)) = T(s(\alpha,\beta,\gamma,J)) = \infty$
for Lebesgue almost every $(\beta , \gamma) \in \T^2$.
\end{theorem}

\begin{proof}
Let $r_1, r_2, r_3, \ldots$ be a sequence of rational numbers that
contains each fixed $\frac{p}{q} \in \Q \cap (0,1)$ infinitely
many times.

Fix $k$ and consider the coding of $r_k n^2 + \beta n + \gamma$
with respect to the given partition, denoted by $s_k$. By
Theorem~\ref{ratreplem} we have that
$
T(s_k) = \infty \text{ for almost every } \beta,\gamma.
$

For $m \in \Z_+$, we can therefore choose a set $\mathcal{B}_{m,k}
\subset \T^2$ such that
\begin{itemize}

\item $B_{m,k}$ is open,

\item $\mathrm{Leb} (\mathcal{B}_{m,k}) > 1 - \frac{1}{2^m}$,

\item for $(\beta , \gamma) \in \mathcal{B}_{m,k}$, $s_k$ has $m$
repetitions in both directions at least once,

\item for $(\beta , \gamma) \in \mathcal{B}_{m,k}$, the itinerary
$r_k n^2 + \beta n + \gamma$ does not contain any partition point
for $n$'s from the finite interval on which we observe the $m$
repetitions in both directions.

\end{itemize}

Next, choose $\mathcal{K}_{m,k} \subset \mathcal{B}_{m,k}$ such
that
\begin{itemize}

\item $\mathcal{K}_{m,k}$ is compact,

\item $\mathrm{Leb} (\mathcal{K}_{m,k}) > 1 - \frac{1}{2^m}$.

\end{itemize}

By compactness of $\mathcal{K}_{m,k}$, we have that for
$(\beta,\gamma) \in \mathcal{K}_{m,k}$, the itinerary $r_k n^2 +
\beta n + \gamma$ for $n$'s from the finite interval on which we
observe the $m$ repetitions in both directions has a uniform
positive distance from the partition points.

Consequently, we can perturb $r_k$ slightly and not change the
coding on the (large) finite interval that supports the two-sided
repetition in question. In other words, there is an open set
$\mathcal{U}_k$ containing $r_k$ such that for $\alpha \in
\mathcal{U}_k$ and $(\beta,\gamma) \in \mathcal{K}_{m,k}$, $s$ has
$m$ repetitions in both directions at least once.

Define
$
\mathcal{R = }\bigcap_{m \ge 1} \bigcup_{k \ge m} \mathcal{U}_k.
$
Since $r_1, r_2, r_3, \ldots$ contains each fixed $\frac{p}{q} \in
\Q \cap (0,1)$ infinitely many times, the open set $\bigcup_{k \ge
m} \mathcal{U}_k$ is dense.

Therefore, $\mathcal{R}$ is a dense  $G_\delta$ set.

By construction, we have that for $\alpha \in \mathcal{R}$ and
Lebesgue almost every $(\beta,\gamma) \in \T^2$, $R(s) = T(s) =
\infty$. Indeed, if $\alpha \in \mathcal{R}$, then $\alpha$
belongs to some $\mathcal{U}_k$. By Borel-Cantelli and the measure
estimates for the sets $\mathcal{K}_{m,k}$, we have that for
almost every $(\beta,\gamma) \in \T^2$, $s$ has $m$ repetitions to
both sides for any $m$.
\end{proof}

The residual set obtained in Theorem~\ref{propwellapprox} is not
explicit. If we are willing to settle for infinite repetitions for
just one pair $(\beta,\gamma) \in \T^2$, then the following result
is of interest. It also has the advantage that the argument we
give can treat general polynomials, not merely quadratic
polynomials, and hence we state and prove the result in this more
general setting. For $\tau > 0$, denote
$$
\mathcal{S}_\tau = \{ \alpha \in \T : \langle q \alpha \rangle <
q^{-\tau} \text{ for infinitely many odd positive integers } q \}.
$$
Clearly, $\mathcal{S}_\tau$ is a residual subset of $\T$.

\begin{theorem}
Let $r$ be a positive integer and $\varepsilon > 0$. Then, for
every $\alpha \in \mathcal{S}_{r + \varepsilon}$, the sequence $s$
given by $s_n = \chi_{[0,1/2)}(\alpha n^r + 1/4)$ obeys $R(s) =
T(s) = \infty$.
\end{theorem}

\begin{proof}
Fix an integer $m \ge 2$ and let $q$ be an odd integer with
$\langle q \alpha \rangle < q^{- r - \varepsilon}$. Let $p \in \Z$
be such that $\langle q\alpha \rangle = |q\alpha - p|$. Then, for
$|n| \le mq$ we have, on the one hand,
\begin{align*}
\langle \alpha (n+q)^r - \alpha n^r \rangle & = \left\langle
\sum_{j = 0}^{r-1} \begin{pmatrix} r \\ j \end{pmatrix} \alpha n^j
q^{r-j} \right\rangle \le \sum_{j = 0}^{r-1} \begin{pmatrix} r
\\ j
\end{pmatrix} n^j q^{r - 1 - j} \left\langle q \alpha
\right\rangle \\
& \le \sum_{j = 0}^{r-1} \begin{pmatrix} r \\ j
\end{pmatrix} m^j q^{r-1} \frac{1}{q^{r + \varepsilon}}
= \frac{1}{q^{1+\varepsilon}} \sum_{j = 0}^{r-1} \begin{pmatrix} r
\\ j \end{pmatrix} m^j
\end{align*}
and, on the other hand,
$$
\langle \alpha n^r \pm 1/4 \rangle \ge \left\langle \frac{p}{q}
n^r \pm \frac{1}{4} \right\rangle - \frac{n^r}{q^{1 + r +
\varepsilon}} \ge \frac{1}{4q} - \frac{m^r}{q^{1+\varepsilon}}.
$$
Since for large enough $q$, we have that
$$
\frac{1}{4q} - \frac{m^r}{q^{1+\varepsilon}} >
\frac{1}{q^{1+\varepsilon}} \sum_{j = 0}^{r-1} \begin{pmatrix} r
\\ j \end{pmatrix} m^j,
$$
and hence $T(s) \ge m$. Since $m$ was arbitrary, this shows $T(s)
= \infty$, which also implies that $R(s) = \infty$.
\end{proof}

\end{document}